\documentclass[conference]{IEEEtran}
\IEEEoverridecommandlockouts
\usepackage{amsthm}
\usepackage{amsfonts}
\usepackage{mathtools}

\usepackage{color}
\usepackage[normalem]{ulem}
\usepackage{amsmath}
\usepackage{hyperref} 
\usepackage{subfig}

\usepackage{cite}

\def\*#1{\mathbf{#1}}
\newtheorem{thm}{Theorem}[section]
\newtheorem{assum}[thm]{Assumption}
\newtheorem{lem}[thm]{Lemma}
\newtheorem{cor}[thm]{Corollary}

\newtheorem{defi}{Definition}
\newcommand{\R}{\mathbb{R}}
\newcommand{\mL}{\mathcal{L}}

\ifCLASSINFOpdf
\else
\fi
\begin{document}
%
\title{On the Convergence of NEAR-DGD for Nonconvex Optimization with Second Order Guarantees}

\author{\IEEEauthorblockN{Charikleia Iakovidou and Ermin Wei}
\IEEEauthorblockA{Department of Electrical and Computer Engineering\\
Northwestern University\\
Evanston, IL \\
Email: chariako@u.northwestern.edu, ermin.wei@northwestern.edu}
\thanks{This work was supported by NSF NRI 2024774.}
}


%


\maketitle

\begin{abstract}
We consider the setting where the nodes of an undirected, connected network collaborate to solve a shared objective modeled as the sum of smooth functions. We assume that each summand is privately known by a unique node. NEAR-DGD is a distributed first order method which permits adjusting the amount of communication between nodes relative to the amount of computation performed locally in order to balance convergence accuracy and total application cost. In this work, we generalize the convergence properties of a variant of NEAR-DGD from the strongly convex to the nonconvex case. Under mild assumptions, we show convergence to minimizers of a custom Lyapunov function. Moreover, we demonstrate that the gap between those minimizers and the second order stationary solutions of the original problem can become arbitrarily small depending on the choice of algorithm parameters. Finally, we accompany our theoretical analysis with a numerical experiment to evaluate the empirical performance of NEAR-DGD in the nonconvex setting.
\end{abstract}



%
\IEEEpeerreviewmaketitle

\section{Introduction}
We focus on optimization problems where the cost function can be modeled as a summation of $n$ components,
\begin{equation}
\label{eq:prob_orig}
    \min_{x \in \mathbb{R}^p} f(x) = \sum_{i=1}^n f_i(x),
\end{equation}
where  $f: \mathbb{R}^p \rightarrow \mathbb{R}$ is a smooth and (possibly) nonconvex function.

Problems of this type frequently arise in a variety of decentralized systems such as wireless sensor networks, smart grids and systems of autonomous vehicles.
A special case of this setting involves a connected, undirected network of $n$ nodes $\mathcal{G}(\mathcal{V},\mathcal{E})$, where $\mathcal{V}$ and $\mathcal{E}$ denote the sets of nodes and edges, respectively. Each node $i \in \mathcal{V}$ has private access to the cost component $f_i$ and maintains a local estimate $x_i$ of the global decision variable $x$. This leads to the following equivalent reformulation of Problem~\eqref{eq:prob_orig},
\begin{equation}
    \label{eq:consensus_prob}
    \begin{split}
    \min_{x_i \in \mathbb{R}^{p}} \sum_{i=1}^n f_i(x_i), \quad \text{s.t. } x_i=x_j, \quad \forall(i,j) \in \mathcal{E}.
    \end{split}
\end{equation}
One of the first algorithms proposed for the solution of Problem~\eqref{eq:consensus_prob} when the functions $f_i$ are convex is the Distributed (Sub)Gradient Descent (DGD) method~\cite{NedicSubgradientConsensus}, which relies on the combination of two elements: $i$) local gradient steps on the functions $f_i$ and $ii$) calculations of weighted averages of local and neighbor variables $x_i$. For the remainder of this work, we will be referring to these two procedures as \emph{computation} and \emph{consensus} (or \emph{communication}) steps, respectively. While DGD has been shown to converge to an approximate solution of Problem~\eqref{eq:consensus_prob} under constant steplengths, a subset of methods known as gradient tracking algorithms  ~\cite{extra,diging,harnessing_smoothness} overcomes this limitation by iteratively estimating the average gradient between nodes.

The convergence of DGD when the function $f$ is nonconvex has been studied in~\cite{zeng_nonconvex_2018}. NEXT~\cite{di_lorenzo_next_2016}, SONATA~\cite{scutari_distributed_2019,daneshmand_second-order_2020}, xFilter~\cite{sun_distributed_2019} and MAGENTA~\cite{hong_divergence_2020}, are some examples of distributed methods that utilize gradient tracking and can handle nonconvex objectives. Other approaches include primal-dual algorithms~\cite{hong_distributed_2018,hong_gradient_2018} (we note that primal-dual and gradient tracking algorithms are equivalent in some cases~\cite{diging}), the perturbed push-sum method~\cite{tatarenko_non-convex_2016}, zeroth order methods~\cite{hajinezhad_zone_2019,tang_distributed_2020}, and stochastic gradient algorithms~\cite{bianchi_convergence_2011,lian_can_nodate,tang_d2_2018,hong_semi_grad}.

Providing second order guarantees when Hessian information is not available is a challenging task. As a result, the majority of the works listed in the previous paragraph establish convergence to critical points only. A recent line of research leverages existing results from dynamical systems theory and the structural properties of certain problems (which include matrix factorization, phase retrieval and dictionary learning, among others) to demonstrate that several centralized first order algorithms  converge to minimizers almost surely when initialized randomly~\cite{lee2017firstorder}. Specifically, if the objective function satisfies the \emph{strict saddle} property, namely, if all critical points are either strict saddles or minimizers, then many first order methods converge to saddles only if they are initialized in a low-dimensional manifold with measure zero. 
Using similar arguments, almost sure convergence to second order stationary points of Problem~\eqref{eq:consensus_prob} is proven in~\cite{daneshmand_second-order_2020} for DOGT, a gradient tracking algorithm for directed networks, and in~\cite{hong_gradient_2018} for the first order primal-dual algorithms GPDA and GADMM. The convergence of DGD with constant steplength to a neighborhood of the minimizers of Problem~\eqref{eq:consensus_prob} is also shown in~\cite{daneshmand_second-order_2020}. The conditions under which the Distributed Stochastic Gradient method (D-SGD), and Distributed Gradient Flow (DGF), a continuous-time approximation of DGD, avoid saddle points are studied in~\cite{swenson2020distributed_j} and~\cite{swenson2020distributed_c}, respectively. Finally, the authors of~\cite{tatarenko_non-convex_2016} prove almost sure convergence to local minima under the assumption that the objective function has no saddle points. 

Given the diversity of distributed systems in terms of computing power, connectivity and energy consumption, among other concerns, the ability to adjust the relative amounts of communication and computation on a case-by-case basis is a desirable attribute for a distributed optimization algorithm. While some existing methods are designed to minimize overall communication load (for instance, the authors of~\cite{sun_distributed_2019} employ Chebyshev polynomials to improve communication complexity), all of the methods listed above perform fixed amounts of computation and communication at every iteration and lack adaptability to heterogeneous environments.

\subsection{Contributions}
In this work, we extend the convergence analysis of the NEAR-DGD method, originally proposed in~\cite{berahas_balancing_2019}, from the strongly convex to the nonconvex setting. NEAR-DGD is a distributed first order method with a flexible framework, which allows for the exchange of computation with communication in order to reach a target accuracy level while simultaneously maintaining low overall application cost. We design a custom Lyapunov function which captures both progress on Problem~\eqref{eq:prob_orig} and distance to consensus, and demonstrate that under relatively mild assumptions, a variant of NEAR-DGD converges to the set of critical points of this Lyapunov function and to approximate critical points of the function $f$ of Problem~\eqref{eq:prob_orig}. Moreover, we show that the distance between the limit points of NEAR-DGD and the critical points of $f$ can become arbitrarily small by appropriate selection of algorithm parameters. Finally, we employ recent results based on dynamical systems theory to prove that the same variant of NEAR-DGD almost surely avoids the saddle points of the Lyapunov function when initialized randomly. Our analysis is shorter and simpler compared to other works due to the convenient form of our Lyapunov function. The implication is that NEAR-DGD asymptotically converges to second order stationary solutions of Problem~\eqref{eq:prob_orig} as the values of algorithm parameters increase.


\subsection{Notation}
In this paper, all vectors are column vectors. We will use the notation $v'$ to refer to the transpose of a vector $v$. The concatenation of local vectors $v_{i}  \in \mathbb{R}^p$ is denoted by $\mathbf{v} = [v_i',...,v_n']' \in \mathbb{R}^{np}$ with a lowercase boldface letter. The average of the vectors $v_1,...,v_n$ contained in $\*v$ will be denoted by $\bar{v}$, i.e. $\bar{v}=\frac{1}{n}\sum_{i=1}^n v_i$. We use uppercase boldface letters for matrices and will denote the element in the $i^{th}$ row and $j^{th}$ column of matrix $\*H$ with $h_{ij}$. We will refer to the $i^{th}$ (real) eigenvalue in ascending order (i.e. 1st is the smallest) of a matrix $\*H$ as $\lambda_i(\*H)$. We use the notations $I_p$ and $1_n$ for  the identity matrix of dimension $p$ and the vector of ones of dimension $n$, respectively.
We will use $\|\cdot\|$ to denote the $l_2$-norm, i.e. for $v \in \mathbb{R}^p$ we have $\left\|v\right\|= \sqrt{\sum_{i=1}^p \left[v\right]_i^2}$ where $\left[v\right]_i$ is the $i$-th element of $v$. The inner product of vectors $v,u$ will be denoted by $\langle v, u \rangle$.  The symbol $\otimes$ will denote the Kronecker product operation. Finally, we define the averaging matrix $\*M:=\left(\frac{1_n 1'_n}{n} \otimes I_p\right)$.


\subsection{Organization}

The rest of this paper is organized as follows. We briefly review the NEAR-DGD method and list our assumptions for the rest of this work in Section~\ref{sec:method}. We analyze the convergence properties of NEAR-DGD when the function $f$ of Problem~\eqref{eq:prob_orig} is nonconvex in Section~\ref{sec:analysis}. Finally, we present the results of a numerical experiment we conducted to assess the empirical performance of NEAR-DGD in the nonconvex setting in Section~\ref{sec:numerical}, and conclude this work in Section~\ref{sec:conclusion}.

\section{The NEAR-DGD method}
\label{sec:method}
In this section, we list our assumptions for the remainder of this work and briefly review the NEAR-DGD method, originally proposed for strongly convex optimization in~\cite{berahas_balancing_2019}. We first introduce the following compact reformulation of Problem~\eqref{eq:consensus_prob},
\begin{equation}
    \label{eq:consensus_prob2}
    \begin{split}
    \min_{\*x \in \mathbb{R}^{np}} \*f\left(\*x\right):=\sum_{i=1}^n f_i(x_i), &\quad \text{s.t. } \left(\*W \otimes I_p \right)\*x=\*x,
    \end{split}
\end{equation}
where $\*x=[x_1',...,x_n']'$ in $\mathbb{R}^{np}$ is the concatenation of the local variables $x_i$, $\*f:\mathbb{R}^{np}\rightarrow \mathbb{R}$ and $\*W \in \mathbb{R}^{n \times n}$ is a matrix satisfying the following condition.
\begin{assum} \textbf{(Consensus matrix)}
\label{assum:symmetry} The matrix $\*W \in \mathbb{R}^{n \times n}$
has the following properties: i) symmetry, ii) double stochasticity, iii) $w_{ij} > 0$ if and only if $(ij) \in \mathcal{E}$ or $i=j$ and $w_{ij} = 0$ otherwise and iv) positive-definiteness.
\end{assum}
We can construct a matrix $\tilde{\*W}$ satisfying properties $(i)-(iii)$ of Assumption~\ref{assum:symmetry} by defining its elements to be max degree or Metropolis-Hastings weights~\cite{XIAO200733}, for instance. Such matrices are not necessarily positive-definite, so we can further enforce property $(iv)$ using simple linear transformations (for example, we could define $\*W = (1-\delta)^{-1}(\tilde{\*W} - \delta I_n)$, where $\delta < \lambda_1(\tilde{\*W})$ is a constant). For the rest of this work, we will be referring to the $2^{nd}$ largest eigenvalue of $\*W$ as $\beta$, i.e. $\beta = \lambda_{n-1}(\*W)$.


We adopt the following standard assumptions for the global function~$\*f$ of Problem~\ref{eq:consensus_prob2}.
\begin{assum} \textbf{(Global Lipschitz gradient)}
\label{assum:smoothness}
The global objective function $\*f:\mathbb{R}^{np} \rightarrow \mathbb{R}$ has \mbox{$L$-Lipschitz} continuous gradients, i.e. $ \|\nabla \*f(\*x) - \nabla \*f(\*y)\| \leq L\|\*x-\*y\|,\quad \forall \*x,\*y \in \mathbb{R}^{np}$.
\end{assum}

\begin{assum}\textbf{(Coercivity)}
\label{assum:coercive}
The global objective function $\*f:\mathbb{R}^{np} \rightarrow \mathbb{R}$ is coercive, i.e. $\lim_{k \rightarrow \infty} \*f(\*x_k) = \infty$ for every sequence $\{\*x_k\}$ such that $\|\*x_k\| \rightarrow \infty$.
\end{assum}


\subsection{The NEAR-DGD method}
Starting from (arbitrary) initial points $y_{i,0}=x_{i,0}$, the local iterates of NEAR-DGD at node $i \in \mathcal{V}$ and iteration count $k$ can be expressed as,
\begin{subequations}
\begin{align}
&x_{i,k} = \sum_{j=1}^n w^{t(k)}_{ij} y_{j,k},\label{eq:near_dgd_local_x_ORIG}\\
&y_{i,k+1}= x_{i,k} - \alpha \nabla f_i \left(x_{i,k}\right),\label{eq:near_dgd_local_y_ORIG}
\end{align}
\end{subequations}
where $\{t(k)\}$ is a predefined sequence of consensus rounds per iteration, $\alpha>0$ is a positive steplength and $w_{ij}^{t(k)}$ is the element in the $i^{th}$ row and $j^{th}$ column of the matrix $\*W^{t(k)}$, resulting from the composition of $t(k)$ consensus operations,
\begin{equation*}
    \*W^{t(k)}=\underbrace{\*W \cdot \*W \cdot ... \cdot \*W}_{t(k)\text{ times}} \in \mathbb{R}^{n \times n}.
\end{equation*}
The system-wide iterates of NEAR-DGD can be written as,
\begin{subequations}
\begin{align}
&\*x_k = \*Z^{t(k)} \*y_k \label{eq:near_dgd_x}\\
    &\*y_{k+1} = \*x_{k} - \alpha \nabla \*f (\*x_{k}) \label{eq:near_dgd_y},
\end{align}
\end{subequations}
where $\*x_k=[x_{1,k}',...,x_{n,k}']' \in \mathbb{R}^{np}$, $\*y_{k+1}=[y_{1,k+1}',...,y_{n,k+1}']' \in \mathbb{R}^{np}$ and $\*Z^{t(k)} = (\*W^{t(k)} \otimes I_p) \in \mathbb{R}^{np \times np}$.

The sequence of consensus rounds per iteration $\{t(k)\}$ can be suitably chosen  to balance convergence accuracy and total cost on a per-application basis. When the functions $f_i$ are strongly convex, NEAR-DGD paired with an increasing sequence $\{t(k)\}$ converges to the optimal solution of Problem~\eqref{eq:consensus_prob2}, and achieves exact convergence with geometric rate (in terms of gradient evaluations) when $t(k)=k$~\cite{berahas_balancing_2019}.

\section{Convergence Analysis}
\label{sec:analysis}
In this section, we present our theoretical results on the convergence of a variant of NEAR-DGD where the number of consensus steps per iteration is fixed, i.e. $t(k)=t$ in~\eqref{eq:near_dgd_local_x_ORIG} and~\eqref{eq:near_dgd_x} for some $t \in \mathbb{N}^+$. We will refer to this method as NEAR-DGD$^t$. First, we introduce the following Lyapunov function, which will play a key role in our analysis,
\begin{equation}
\label{eq:lyapunov_f}
   \mathcal{L}_t\left(\*y\right)= \*f\left(\*Z^t\*y\right) +\frac{1}{2\alpha}\left\|\*y\right\|^2_{\*Z^t}-\frac{1}{2\alpha}\left\|\*y\right\|^2_{\*Z^{2t}}.
\end{equation}
Using~\eqref{eq:lyapunov_f}, we can express the $\*x_k$ iterates of NEAR-DGD$^t$ as,
\begin{equation}
\label{lem:grad_x}
    \begin{split}
        \*x_{k+1}= \*x_k - \alpha \nabla\mathcal{L}_t\left(\*y_k\right).
    \end{split}
\end{equation}
We need one more assumption on the geometry of the Lyapunov function $\mL_t$ to guarantee the convergence of NEAR-DGD$^t$. Namely, we require $\mL_t$  to be "sharp" around its critical points, up to a reparametrization. This property is formally summarized below.
\begin{defi}
\label{def:kl_prop}
\textbf{(Kurdyka-\L ojasiewicz (KL) property)}~\cite{attouch_convergence_2013} A function $h:\mathbb{R}^p\rightarrow \mathbb{R} \cup \{+\infty\}$ has the KL property at $x^\star \in \text{dom}(\partial h)$ if there exists $\eta \in (0,+\infty]$, a neighborhood $U$ of $x^\star$, and a continuous concave function $\phi: [0,\eta)\rightarrow \mathbb{R}^+$ such that: $i$) $\phi(0) = 0$, $ii$) $\phi$ is $\mathcal{C}^1$ (continuously differentiable) on $(0,\eta)$, $iii$) for all $s \in (0,\eta)$, $\phi^\prime (s)>0$, and $iv$) for all $x \in U \cap \{x: h(x^\star)< h(x)< h(x^\star) +\eta\}$, the KL inequality holds:
    \begin{equation*}
    \label{eq:kl_inequality}
        \phi^\prime (h(x)-h(x^\star)) \cdot \text{dist}(0,\partial h(x)) \geq 1.
    \end{equation*}
Proper lower 
semicontinuous functions which satisfy the KL inequality at each point of $\text{dom}(\partial h)$ are called KL functions.
\end{defi}
\begin{assum}\textbf{(KL Lyapunov function)}
\label{assum:kl}
The Lyapunov function $\mL_t:\mathbb{R}^{np} \rightarrow \mathbb{R}$ is a KL function.
\end{assum}
Assumption~\ref{assum:kl} covers a broad range of functions, including real analytic, semialgebraic and globally subanalytic functions (see~\cite{attouch_proximal_2013} for more details). For instance, if the function $\*f$ is real analytic, then $\mL_t$ is the sum of real analytic functions and by extension KL.

\subsection{Convergence to critical points}



In this subsection, we demonstrate that the $\*y_k$ iterates of NEAR-DGD$^t$ converge to a critical point of the Lyapunov function $\mL_t$~\eqref{eq:lyapunov_f}. We assume that Assumptions~\ref{assum:symmetry}-\ref{assum:coercive} and Assumption~\ref{assum:kl} hold for the rest of this work. We begin our analysis by showing that the sequence $\{\mL_t(\*y_k)\}$ is non-increasing in the following Lemma. 
\begin{lem} \textbf{(Sufficient Descent)}
\label{lem:suff_desc}
Let \{$\*y_k$\} be the sequence of NEAR-DGD$^t$ iterates generated by~\eqref{eq:near_dgd_y} and suppose that the steplength $\alpha$ satisfies $\alpha < 2/L$, where $L$ is defined in Assumption~\ref{assum:smoothness}. Then the following inequality holds for the sequence $\{\mL_t(\*y_k)\}$,
\begin{equation*}
    \begin{split}
      \mathcal{L}_t\left(\*y_{k+1}\right) &\leq \mathcal{L}_t\left(\*y_k\right) - \rho \left \| \*y_{k+1}-\*y_k\right\|^2,
    \end{split}
\end{equation*}
where $\rho=(2\alpha)^{-1}\min_i \left(\lambda^t_i(\*Z)\left(1+\left(1-\alpha L\right)\lambda^t_i(\*Z)\right)\right) > 0$.
\end{lem}

\begin{proof} Combining~\eqref{eq:near_dgd_x} and~\eqref{eq:lyapunov_f}, we obtain for $k\geq0$,
\begin{equation}
\label{eq:suff_desc1}
\begin{split}
     \mL_t(\*y_{k})= \*f(\*x_{k}) + \frac{1}{2\alpha}\langle \*y_{k},\*x_{k}\rangle - \frac{1}{2\alpha}\|\*x_{k}\|^2.
     \end{split}
\end{equation}

Let $\*d_x := \*x_{k+1}-\*x_k$. Assumption~\ref{assum:smoothness} then yields $\*f(\*x_{k+1})
        \leq \*f(\*x_k) + \langle \nabla \*f(\*x_k),\*d_x\rangle + \frac{L}{2}\|\*d_x\|^2 = \*f(\*x_k) - \frac{1}{\alpha}\langle  \*y_{k+1}-\*x_k,\*d_x\rangle + \frac{L}{2}\|\*d_x\|^2,$
where we acquire the last equality from~\eqref{eq:near_dgd_y}. Substituting this relation in~\eqref{eq:suff_desc1} applied at the $(k+1)^{th}$ iteration, we obtain,
\begin{equation*}
    \begin{split}
        &\mL_t(\*y_{k+1})
        \leq \*f(\*x_k) - \frac{1}{2\alpha}\langle  \*y_{k+1}-\*x_k,\*d_x\rangle + \frac{L}{2}\|\*d_x\|^2 \\
        &\quad+ \frac{1}{2\alpha}\langle \*y_{k+1},\*x_{k+1}\rangle - \frac{1}{2\alpha}\|\*x_{k+1}\|^2\\
        &= \mL_t(\*y_k) - \frac{1}{2\alpha}\langle \*y_{k},\*x_{k}\rangle + \frac{1}{2\alpha}\|\*x_{k}\|^2- \frac{1}{\alpha}\langle  \*y_{k+1}-\*x_k,\*d_x\rangle \\
        &\quad + \frac{L}{2}\|\*d_x\|^2+ \frac{1}{2\alpha}\langle \*y_{k+1},\*x_{k+1}\rangle - \frac{1}{2\alpha}\|\*x_{k+1}\|^2,
    \end{split}
\end{equation*}
    where we obtain the equality after further application of~\eqref{eq:suff_desc1}.
After setting $\*d_y := \*y_{k+1}-\*y_k$ and re-arranging the terms, we obtain,
\begin{equation*}
    \begin{split}
        \mL_t(\*y_{k+1})
        &\leq \mL_t(\*y_k) - \frac{1}{2\alpha}\langle \*y_{k},\*x_{k}\rangle + \frac{1}{\alpha}\langle  \*y_{k+1},\*x_k\rangle \\
        &\quad - \frac{1}{2\alpha}\langle \*y_{k+1},\*x_{k+1}\rangle - \left(\frac{1}{2\alpha} - \frac{L}{2}\right)\|\*d_x\|^2 \\
        &=\mL_t(\*y_k) - \frac{1}{2\alpha}\| \*y_{k}\|^2_{\*Z^t} + \frac{1}{\alpha}\langle  \*y_{k+1},\*y_k\rangle_{\*Z^t} \\
        &\quad - \frac{1}{2\alpha}\| \*y_{k+1}\|^2_{\*Z^t} - \left(\frac{1}{2\alpha} - \frac{L}{2}\right)\|\*d_x\|^2 \\
        &=\mL_t(\*y_k) - \frac{1}{2\alpha}\| \*d_y\|^2_{\*Z^t} - \left(\frac{1}{2\alpha} - \frac{L}{2}\right)\|\*d_x\|^2.
    \end{split}
\end{equation*}
Let $\*H:=(2\alpha)^{-1}\*Z^t \left(I + (1-\alpha L)\*Z^t\right)$, which is a positive definite matrix due to Assumption~\ref{assum:symmetry} and the fact that $\alpha <2/L$. Moreover, $\|\*d_x\|^2 = \|\*d_y\|^2_{\*Z^{2t}}$ by Eq.~\eqref{eq:near_dgd_x}. We can then re-write the immediately previous relation as $\mL_t(\*y_{k+1})
        \leq \mL_t(\*y_k)  -\| \*y_{k+1}-\*y_k\|^2_{\*H}$.
Applying the definition of $\rho= \lambda_1(\*H)$ concludes the proof.
\end{proof}
An important consequence of Lemma~\ref{lem:suff_desc} is that NEAR-DGD$^t$ can tolerate a bigger range of steplengths than previously indicated ($\alpha < 2/L$ vs. $\alpha < 1/L$ in~\cite{berahas_balancing_2019}). Moreover, Lemma~\ref{lem:suff_desc} implies that the sequence $\{\mL_t(\*y_k)\}$ is upper bounded by $\mL_t(\*y_0)$. We use this fact to prove that the iterates of NEAR-DGD$^t$ are also bounded in the next Lemma.
\begin{lem}\textbf{(Boundedness)}
\label{lemma:bounded}
Let $\{\*x_k\}$ and $\{\*y_k\}$ be the sequences of NEAR-DGD$^t$ ($t(k)=t$) iterates generated by~\eqref{eq:near_dgd_x} and~\eqref{eq:near_dgd_y}, respectively, from initial point $\*y_0$ and under steplength $\alpha < 2/L$. Then the following hold: $i$) the sequence $\{\mathcal{L}_t(\*y_k)\}$ is lower bounded, and $ii$)
there exist universal positive constants $B_x$ and $B_y$ such that $\|\*x_k\|\leq B_x$ and $\|\*y_{k+1}\|\leq B_y$ for all $k\geq0$ and $t \in \mathbb{N}^+$.
\end{lem}

\begin{proof} 
By Assumption~\ref{assum:coercive}, the function $\*f$ is lower bounded and therefore $\mL_t$ is also lower bounded (sum of lower bounded functions). This proves the first claim of this Lemma.

To prove the second claim, we first notice that Lemma~\ref{lem:suff_desc} implies that the sequence $\{\mathcal{L}_t(\*y_k)\}$ is upper bounded by $\mL_t(\*y_0)$. Let us define the set $\mathcal{X}_0 := \{\*Z^t \*y_0, t \in \mathbb{N}^+\}$. The set $\mathcal{X}_0$ is compact, since $\|\*Z^t \*y_0\| \leq \|\*y_0\|$ for all $t \in \mathbb{N}^+$ due to the non-expansiveness of $\*Z$. Hence, by the continuity of $\*f$ and the Weierstrass Extreme Value Theorem, there exists $\hat{\*x}_0 \in \mathcal{X}_0$ such that $ \*f(\*x_0) \leq \*f(\hat{\*x}_0)$ for all $\*x_0 \in \mathcal{X}_0$. Moreover, Assumption~\ref{assum:symmetry} yields $\|\*y_0\|^2_{\*Z^t(I-\*Z^t)} \leq \|\*y_0\|^2$ for all positive integers $t$, and therefore $\mL_t(\*y_0) \leq \hat{\mL}$ for all $t \in \mathbb{N}^{+}$, where $\hat{\mL} = \*f(\hat{\*x}_0) + (2\alpha)^{-1}\|\*y_0\|^2$.

Since $\hat{\mL} \geq \mL_t(\*y_0) \geq \mL_t(\*y_k) \geq \*f(\*Z^t \*y_k) = \*f(\*x_k)$ for all $k\geq0$ and $t>0$, the sequence $\{\*f(\*x_k)\}$ is upper bounded. Hence, by Assumption~\ref{assum:coercive}, there exists positive constant $B_x$ such that $\|\*x_k\|\leq B_x$ for $k\geq0$ and $t>0$. Moreover, Assumption~\ref{assum:smoothness} yields $ \*f(\*y_{k+1}) \leq \*f(\*x_k) + \langle \nabla \*f(\*x_k),\*y_{k+1}-\*x_k\rangle + \frac{L}{2}\|\*y_{k+1}-\*x_k\|^2=\*f(\*x_k) - \alpha\|\nabla \*f(\*x_k)\|^2 + \frac{\alpha^2 L }{2}\|\nabla \*f(\*x_k)\|^2= \*f(\*x_k) - \alpha \left(1-\frac{\alpha L}{2}\right)\|\nabla \*f(\*x_k)\|^2 \leq \*f(\*x_k),$
where we obtain the first equality from~\eqref{eq:near_dgd_y} and last inequality from the fact that $\alpha < 2/L$. This relation
combined with Assumption~\ref{assum:coercive} implies that there exists constant $B_y>0$ such that $\|\*y_{k+1}\|\leq B_y$ for $k>0$ and $t>0$, which concludes the proof.
\end{proof}
Next, we use Lemma~\ref{lemma:bounded} to show that the distance between the local iterates generated by NEAR-DGD$^t$ and their average is bounded. 
\begin{lem}\textbf{(Bounded distance to mean)}
\label{lemma:bounded_dist}
Let $x_{i,k}$ and $y_{i,k}$ be the local NEAR-DGD$^t$ iterates produced under steplength $\alpha < 2/L$ by~\eqref{eq:near_dgd_local_x_ORIG} and~\eqref{eq:near_dgd_local_y_ORIG}, respectively, and define the average iterates $\Bar{x}_k := \frac{1}{n}\sum_{i=1}^n x_{i,k}$ and  $\Bar{y}_k := \frac{1}{n}\sum_{i=1}^n y_{i,k}$. Then the distance between the local and the average iterates is bounded for $i=1,...,n$ and $k=1,2,...$, i.e.
\begin{equation*}
     \left\|x_{i,k}-\Bar{x}_k \right\| \leq \beta^t B_y,\text{ and } \left\|y_{i,k}-\Bar{y}_k \right\| \leq B_y,
\end{equation*}
where $B_y$ is a positive constant defined in Lemma~\ref{lemma:bounded}.
\end{lem}

\begin{proof}
Multiplying both sides of~\eqref{eq:near_dgd_x} with $\*M=\left(\frac{1_n 1_n'}{n} \otimes I_p\right)$ yields $\Bar{x}_k = \Bar{y}_k$. Moreover, we observe that $ \left\|\*v_k- \*M \*v_k \right\|^2 = \sum_{i=1}^n \left\|v_{i,k}-\Bar{v}_k \right\|^2$ for any vector $\*v \in \mathbb{R}^{np}$. Hence,
\begin{equation*}
    \begin{split}
        \left\|x_{i,k}-\Bar{x}_k \right\| &= \left\|x_{i,k}-\Bar{y}_k \right\|
        \leq \left\|\*x_k- \*M \*y_k \right\|\\
        &\leq  \left\|\*Z^t\*y_k- \*M \*y_k \right\|\leq \beta^t \|\*y_k\|,
    \end{split}
\end{equation*}
where we derive the last inequality from the spectral properties of $\*Z$ and $\*M$ (we note that the matrix $1_n 1_n'/n$ has a single non-zero eigenvalue at $1$ associated with the eigenvector $1_{np}$). 

Similarly, for the local iterates $y_{i,k}$ we obtain,
\begin{equation*}
    \begin{split}
        \left\|y_{i,k}-\Bar{y}_k \right\| \leq \left\|\*y_k- \*M \*y_k \right\|=  \left\|\left(I-\*M\right)\*y_k \right\|\leq  \|\*y_k\|.
    \end{split}
\end{equation*}
Applying Lemma~\ref{lemma:bounded} to the two preceding inequalities completes the proof.
\end{proof}
We are now ready to state the first Theorem of this section, namely that there exists a subsequence of $\{\*y_k\}$ that converges to a critical point of $\mL_t$. 
\begin{thm}\textbf{(Subsequence convergence)}
\label{theorem:limit_pts}
Let $\{\*y_k\}$ be the sequence of NEAR-DGD$^t$ iterates generated by~\eqref{eq:near_dgd_y} with steplength $\alpha < 2/L$. Then $\{\*y_k\}$ has a convergent subsequence whose limit point is a critical point of~\eqref{eq:lyapunov_f}.
\end{thm}

\begin{proof}
By Lemma~\ref{lemma:bounded}, the sequence $\{\*y_k\}$ is bounded and therefore there exists a convergent subsequence $\{\*y_{k_s}\}_{s \in \mathbb{N}} \rightarrow \*y^\infty$ as $s \rightarrow \infty$. 
In addition, recursive application of Lemma~\ref{lem:suff_desc} over iterations $0,1,...,k$ yields,
\begin{equation*}
   \mathcal{L}_t\left(\*y_{k}\right) \leq \mathcal{L}_t\left(\*y_0\right) - \rho\sum_{j=0}^{k-1} \left \| \*y_{j+1}-\*y_{j}\right\|^2,
\end{equation*}
where the sequence $\{\mathcal{L}_t\left(\*y_{k}\right)\}$ is non-increasing and bounded from below by Lemmas~\ref{lem:suff_desc} and~\ref{lemma:bounded}.

Hence, $\{\mathcal{L}_k\left(\*y_{k}\right)\}$ converges and the above relation implies that $\sum_{k=1}^{\infty} \left \| \*y_{k+1}-\*y_{k}\right\|^2 < +\infty$ and $ \left \| \*y_{k+1}-\*y_{k}\right\| \rightarrow 0$. Moreover, $ \left \| \*x_{k+1}-\*x_{k}\right\|=\left \| \*y_{k+1}-\*y_{k}\right\|_{\*Z^{2t}} \leq  \left \| \*y_{k+1}-\*y_{k}\right\|$ by the non-expansiveness of $\*Z$ and thus $\|\*x_{k+1}-\*x_k\| \rightarrow 0$. Finally, Eq.~\ref{lem:grad_x} yields $\|\nabla\mathcal{L}_t\left(\*y_k\right)\|=\alpha^{-1}\|\*x_{k+1}-\*x_k\|\rightarrow0$. We conclude that $\left\|\nabla\mathcal{L}_t\left(\*y_{k_s}\right)\right\| \rightarrow0$ as $s\rightarrow \infty$ and therefore $\nabla\mathcal{L}_t\left(\*y^\infty\right)=\mathbf{0}$.
\end{proof}
We note that Assumption~\ref{assum:kl} is not necessary for Theorem~\ref{theorem:limit_pts} to hold. However, Theorem~\ref{theorem:limit_pts} does not guarantee the convergence of NEAR-DGD$^t$; we will need Assumption~\ref{assum:kl}  to prove that NEAR-DGD$^t$ converges in Theorem~\ref{thm:global_conv}. Before that, we introduce the following two preliminary Lemmas that hold only under Assumption~\ref{assum:kl}.

\begin{lem}\textbf{(Bounded difference under the KL property)}
\label{lemma:bounded_diff}
Let $\*x_k$ and $\*y_k$ be the NEAR-DGD$^t$ iterates generated by~\eqref{eq:near_dgd_x} and~\eqref{eq:near_dgd_y}, respectively, under steplength $\alpha < 2/L$.
Moreover, suppose that the KL inequality with respect to some point $\*y^\star \in \mathbb{R}^{np}$ holds at $\*y_k$, i.e.,
 \begin{equation}
 \label{eq:kl_w_grad}
     \phi^\prime(\mL_t(\*y_k) - \mL_t(\*y^\star))\|\mathcal{L}_k(\*y_k)\| \geq 1.
 \end{equation}
Then the following relation holds,
\begin{equation*}
    \begin{split}
         \left\|\*v_{k+1}-\*v_k\right\| \leq \frac{1}{\alpha \rho}\left( \phi\left(l_k\right) - \phi\left(l_{k+1}\right)\right),
    \end{split}
\end{equation*}
where $\|\*v_{k+1}-\*v_k\|$ can be $\|\*x_{k+1}-\*x_k\|$ or $\|\*y_{k+1}-\*y_k\|$ and $l_k:=\mL_t(\*y_k) - \mL_t(\*y^\star)$.
\end{lem}

\begin{proof}
Lemma~\ref{lem:suff_desc} yields $ \rho \left \| \*y_{k+1}-\*y_k\right\|^2 \leq \mathcal{L}_t\left(\*y_k\right) -\mathcal{L}_t\left(\*y_{k+1}\right) = l_k - l_{k+1}$ for $k\geq0$.
We can multiply both sides of this relation with $\phi^\prime \left(l_{k}\right)>0$ to obtain $\rho\phi^\prime \left(l_{k}\right) \left \| \*y_{k+1}-\*y_k\right\|^2 \leq -\phi^\prime \left(l_{k}\right) \left(l_{k+1} - l_{k}\right)\leq \phi\left(l_k\right) - \phi\left(l_{k+1}\right),$
where we derive the last inequality from the concavity of $\phi$. In addition, using Eq.~\ref{lem:grad_x}, we can re-write~\eqref{eq:kl_w_grad} as $\alpha^{-1}\phi^\prime(l_{k})\|\*x_{k+1}-\*x_k\| \geq 1$. Combining these relations, we acquire,
\begin{equation*}
    \begin{split}
        \frac{\alpha\rho \left \| \*y_{k+1}-\*y_k\right\|^2 }{\left\| \*x_{k+1}-\*x_k\right\|} \leq   \phi\left(l_k\right) - \phi\left(l_{k+1}\right).
    \end{split}
\end{equation*}
Observing that $\|\*x_{k+1}-\*x_k\| \leq \|\*y_{k+1}-\*y_k\| $ due to the non-expansiveness of $\*Z$ and re-arranging the terms of the relation above yields the final result.
\end{proof}

In the next Lemma, we show that if NEAR-DGD$^t$ is initialized from an appropriate subset of $\mathbb{R}^{np}$ and Assumption~\ref{assum:kl} holds, then the sequence $\{\*y_k\}$ converges to a critical point of the Lyapunov function $\mL_t$.


\begin{lem}\textbf{(Local convergence)}
\label{lemma:local_conv}
Let $\{\*y_k\}$ be the sequence of iterates generated by~\eqref{eq:near_dgd_y} from initial point $\*y_0$ and with steplength $\alpha < 2/L$. 
Moreover, let $U$ and $\eta$ be the objects in Def.~\ref{def:kl_prop} and suppose that the following relations are satisfied for some point $\*y^\star \in \mathbb{R}^{np}$,
\begin{gather}
(\alpha \rho)^{-1}\phi\left(\mL_t(\*y_0)-\mL_t(\*y^\star)\right)+\| \*y_{0}- \*y^\star\| < r, \label{eq:loc_conv_U} \\
 \mL_t(\*y^\star) \leq \mL_t(\*y_k) < \mL_t(\*y^\star) + \eta, \quad k\geq 0 \label{eq:loc_conv_f},
\end{gather}
where $r$ is a positive constant and $\mathcal{B}(\*y^\star,r) \subset U$.

Then the sequence $\{\*y_k\}$ has finite length, i.e. $\sum_{j=0}^\infty \|\*y_{j+1}-\*y_j\| < \infty$, and converges to a critical point of~\eqref{eq:lyapunov_f}. 
\end{lem}

\begin{proof}
In the trivial case where $\mL_t(\*y_k) = \mL_t(\*y^\star)$, Lemma~\ref{lem:suff_desc} combined with~\eqref{eq:loc_conv_f} yield $\mL_t(\*y_{k+1})=\mL_t(\*y_k)=\mL_t(\*y^\star)$ and $\|\*y_{k+1}-\*y_k\|=0$. Let us now assume that $\mL_t(\*y_k) \in \big(\mL_t(\*y^\star),\mL_t(\*y^\star)+\eta\big)$ and $\*y_k \in \mathcal{B}(\*y^\star,r)$ up to and including some index $\tau \in \mathbb{N}^+$, which implies that~\eqref{eq:kl_w_grad} holds for all $k\leq\tau$. Applying the triangle inequality twice, we obtain,
\begin{equation*}
\begin{split}
    \|\*y_{\tau+1} - \*y^\star\|
    &\leq \|\*y_{\tau+1} - \*y_{0} \|+\| \*y_{0}- \*y^\star\|\\
    &= \left \|\sum_{j=0}^{\tau}\left(\*y_{j+1}-\*y_j\right)\right\| +\| \*y_{0}- \*y^\star\|\\
    &\leq \sum_{j=0}^{\tau}\|\*y_{j+1}-\*y_j\| +\| \*y_{0}- \*y^\star\|.
    \end{split}
\end{equation*}
 Application of Lemma~\ref{lemma:bounded_diff} then yields  $\left \| \*y_{k+1}-\*y_k\right\| \leq (\alpha \rho)^{-1}\left(\phi\left(l_k\right) - \phi\left(l_{k+1}\right)\right)$, for $k \leq \tau$.
Substituting this in the preceding relation, we acquire,
\begin{equation*}
\begin{split}
    \|\*y_{\tau+1} - \*y^\star\|
    &\leq \frac{1}{\alpha \rho} \left(\phi\left(l_{0}\right)-\phi\left(l_{\tau+1}\right)\right)+\| \*y_{0}- \*y^\star\|\\
    &\leq  \frac{\phi\left(l_{0}\right)}{\alpha \rho} +\| \*y_{0}- \*y^\star\| < r.
    \end{split}
\end{equation*}
The above result implies that $\*y_{\tau+1} \in \mathcal{B}(\*y^\star, r)$. Given the fact that that $\|\*y_0 - \*y^\star\| < r$ and thus $\*y_0 \in \mathcal{B}(\*y^\star,r)$, we have $\*y_{k} \in \mathcal{B}(\*y^\star, r)$ and $\left \| \*y_{k+1}-\*y_k\right\| \leq (\alpha \rho)^{-1} \left(\phi\left(l_k\right) - \phi\left(l_{k+1}\right)\right)$ for all $k \geq 0$. Hence,
\begin{equation*}
\begin{split}
    \sum_{j=0}^\infty \|\*y_{j+1} - \*y_j\| &\leq \frac{1}{\alpha \rho} \sum_{j=0}^\infty \left(\phi(l_k)-\phi(l_{k+1})\right)\\
    &\leq \frac{1}{\alpha \rho}\left(\phi(l_0)-\phi(l_\infty)\right)\leq \frac{\phi(l_0)}{\alpha \rho}.
    \end{split}
\end{equation*}
Thus, the sequence $\{\*y_k\}$ is finite and Cauchy (convergent), and $\{\*y_k\} \rightarrow \tilde{\*y}$, where $\tilde{\*y}$ is a critical point of~\eqref{eq:lyapunov_f} by Theorem~\ref{theorem:limit_pts}. 
\end{proof}
Next, we combine our previous results to prove the global convergence of the $\*y_k$ iterates of NEAR-DGD$^t$ in Theorem~\ref{thm:global_conv}.
\begin{thm}\textbf{(Global Convergence)}
\label{thm:global_conv} Let $\{\*y_k\}$ be the sequence of NEAR-DGD$^t$ iterates produced by~\eqref{eq:near_dgd_y} under steplength $\alpha < 2/L$ and let $\*y^\infty$ be a limit point of a convergent subsequence of $\{\*y_k\}$ as defined in Theorem~\ref{theorem:limit_pts}. 

Then under Assumption~\ref{assum:kl} the following statements hold: $i$) there exists an index $k_0 \in \mathbb{N}^+$ such that the KL inequality with respect to $\*y^\infty$ holds for all $k\geq k_0$, and $ii$) the sequence $\{\*y_k\}$ converges to $\*y^\infty$.
\end{thm}

\begin{proof}
We first observe that by Lemma~\ref{lem:suff_desc} the sequence $\{\mL_t(\*y_k)\}$ is non-increasing, and therefore $\mL_t(\*y^\infty) \leq \mL_t(\*y_k)$ for all $k \geq 0$. If Assumption~\ref{assum:kl} holds, then the objects $U$ and $\eta$ in Def.~\ref{def:kl_prop} exist and by the continuity of $\phi$, it is possible to find an index $k_0$ that satisfies the following relations,
\begin{gather*}
    (\alpha \rho)^{-1} \phi\left(\mL_t(\*y_{k_0}) - \mL_t(\*y^\infty)\right) + \|\*y_{k_0} - \*y^\infty\| < r, \\
    \mL_t(\*y_k) \in [\mL_t(\*y^\infty), \mL_t(\*y^\infty)+\eta), \quad \forall k \geq 0,
\end{gather*}
where $\mathcal{B}(\*y^\infty,r) \subset U$.

Applying Lemma~\ref{lemma:local_conv} to the sequence $\{\*y_k\}_{k\geq k_0}$ with $\*y^\star = \*y^\infty$ establishes the convergence of $\{\*y_k\}$. Finally, since $\*y^\infty$ is the limit point of a subsequence of $\{\*y_k\}$ and $\{\*y_k\}$ is convergent, we conclude that $\{\*y_k\}\rightarrow \*y^\infty$.
\end{proof}
Since $\*Z$ is a non-singular matrix, Theorem~\ref{thm:global_conv} implies that the sequence $\{\*x_k\}$ also converges. Moreover, using arguments similar to~\cite{kl_rates}, we can prove the following result on the convergence rate of $\{\*x_k\}$.
\begin{lem}\textbf{(Rates)}
\label{lem:rates}
Let $\{\*x_k\}$ be the sequence of iterates produced by~\eqref{eq:near_dgd_x}, $\*x^\infty = \*Z^t \*y^\infty$ where $\*y^\infty$ is the limit point of the sequence $\{\*y_k\}$ and suppose $\phi(s)=c s^{1-\theta}$ in Assumption~\ref{assum:kl} for some constant $c>0$ and $\theta \in [0,1)$ (for a discussion on $\phi$, we direct readers to~\cite{attouch_proximal_2013}). Then the following hold:
\begin{enumerate}
    \item If $\theta=0$, $\{\*x_k\}$ converges in a finite number of iterations.
    \item If $\theta \in (0,1/2]$, then constants $c>0$  and $Q \in [0,1)$ exist such that $\|\*x_k-\*x^\infty\|\leq cQ^k$.
    \item If $\theta \in (1/2,1)$, then there exists a constant $c>0$ such that $\|\*x_k-\*x^\infty\|\leq c k^{-\frac{1-\theta}{2\theta-1}}$.
\end{enumerate}
\end{lem}
\begin{proof}
$i$) $\theta=0$: From the definition of $\phi$ and $\theta=0$ we have $\phi'(l_k)=c(1-\theta)l_k^{-\theta}=c$. Let $I:=\{k \in \mathbb{N}:\*x_{k+1} \neq \*x_k\}$ (by the non-singularity of $\*Z$, it also follows that $\*y_{k+1}\neq \*y_k$ for $k\in I$). Then for large $k$
the KL inequality holds at $\*y_k$ and we obtain $\|\nabla \mL_t(\*y_k)\|\geq c^{-1}$, or equivalently by~\eqref{lem:grad_x}, $\|\*x_{k+1}-\*x_k\| \geq \alpha c^{-1} $. Application of Lemma~\ref{lem:suff_desc} combined with the fact that $\|\*x_{k+1}-\*x_k\|\leq \|\*y_{k+1}-\*y_k\|$ yields $\mL_t(\*y_{k+1})\leq\mathcal{L}_t\left(\*y_k\right) - \rho \alpha^2 c^{-2}$. Given the convergence of the sequence $\{\mL_t\}$, we conclude that the set $I$ is finite and the method converges in a finite number of steps.

$ii$) $\theta \in (0,1)$: Let $S_k :=\sum_{j=k}^\infty \|\*x_{j+1}-\*x_j\|$ where $\*x^\infty = \*Z^t \*y^\infty$. Since $\|\*x_k - \*x^\infty\| \leq S_k$, it suffices to bound $S_k$. Using Lemma~\ref{lemma:bounded_diff} with $\*y^\star = \*y^\infty$ and for $k\geq k_0$, where $k_0$ is defined in Theorem~\ref{thm:global_conv}, we obtain,
\begin{equation}
\begin{split}
\label{eq:kl_sum_1}
  S_k &\leq \frac{1}{\alpha \rho} \sum_{j=k}^\infty \left(\phi(l_j)-\phi(l_{j+1})\right) =\frac{1}{\alpha \rho} \phi(l_k) = \frac{1}{\nu} l_k^{1-\theta},
    \end{split}
\end{equation}
where $\nu = \alpha \rho / c$.

Moreover, Eq.~\ref{lem:grad_x} yields $\left\|\nabla\mathcal{L}_t\left(\*y_k\right)\right\| = \alpha^{-1} \left\|\*x_{k+1}-\*x_k\right\|= \alpha^{-1}\left(S_k -  S_{k+1}\right)$. Using this relation and the definition of $\phi$, we can express the KL inequality as,
\begin{equation}
\label{eq:kl_sum_2}
     \mu l_k^{-\theta}\left(S_k-S_{k+1}\right) \geq 1,
\end{equation}
where $\mu = \alpha^{-1}c(1-\theta)$.

If $\theta \in (0,1/2]$, raising both sides of the preceding  inequality to the power of $\gamma = \frac{1-\theta}{\theta} > 1$ and re-arranging the terms yields $\mu^\gamma \left(S_k-S_{k+1}\right)^\gamma \geq l_k^{1-\theta}$.
Due to the fact that $S_k - S_{k+1} = \alpha \|\nabla\mathcal{L}_t(\*y_k)\| \rightarrow 0$, there exists some index $k$ such that $S_k - S_{k+1} > \left(S_k - S_{k+1}\right)^\gamma$ and $\mu^\gamma \left( S_k-S_{k+1}\right)\geq l_k^{1-\theta}$. Combining this relation with~\eqref{eq:kl_sum_1}, we obtain $\nu S_k  \leq \mu^\gamma \left( S_k-S_{k+1}\right) \Leftrightarrow S_{k+1} \leq \left(1-\frac{\nu}{\mu^\gamma} \right)S_k.$

If $\theta \in (1/2,1)$, raising both sides of~\eqref{eq:kl_sum_1} to the power of $\theta/(1-\theta)>1$ yields $S_k^{\theta/(1-\theta)} \leq \nu^{-\theta/(1-\theta)} l_k^\theta$. After substituting this relation in~\eqref{eq:kl_sum_2} and re-arranging we obtain $1 \leq C \left(S_k-S_{k+1}\right)(S_k^{\theta/(1-\theta)})^{-1},$
where $C = \mu \nu^{-\theta/(1-\theta)}$. Define $h:(0,+\infty)\rightarrow \mathbb{R}$ to be $h(s)=s^{-\theta/(1-\theta)}$. The preceding relation then yields $1 \leq C (S_k - S_{k+1}) h(S_k) \leq C \int_{S_{k+1}}^{S_k}h(s)ds =  C \zeta^{-1} \left(S_k^{\zeta}-S_{k+1}^{\zeta}\right)$, where $\zeta=(1-2\theta)/(1-\theta) < 0$. After setting $\tilde{C}=-C^{-1}\zeta > 0$ and re-arranging, we obtain $\tilde{C}\leq S^\zeta_{k+1}-S^\zeta_k$. Summing this relation over iterations $k=k_0,...,t-1$ yields $ (t-k_0) \tilde{C} \leq S_{t}^\zeta - S_{k_0}^\zeta \Leftrightarrow S_t \leq \left((t-k_0) \tilde{C}  + S_{k_0}^\zeta\right)^{1/\zeta} \leq c t^{1/\zeta}$,
for some $c > 0$.
\end{proof}

We conclude this subsection with one more result on the distance to optimality of the local $x_{i,k}$ iterates of NEAR-DGD$^t$ and their average $\bar{x}_k=\frac{1}{n}\sum_{i=1}^n x_{i,k}$ as $k\rightarrow \infty$.

\begin{cor}\textbf{(Distance to optimality)}
\label{lemma:dist_optimality}
Suppose that $\{\*y_k\} \rightarrow \*y^\infty$ and let $\*x^\infty=\*Z^t \*y^\infty$. Moreover, let $\bar{x}^\infty = \bar{y}^\infty = \frac{1}{n}\sum_{i=1}^n x^\infty_i$. Then $\bar{x}^\infty$ is an approximate critical point of $f$,
\begin{equation*}
    \left\| \nabla f(\bar{x}^\infty) \right\| \leq \beta^t \sqrt{n} L B_y
\end{equation*}
where $B_y$ is a positive constant defined in Lemma~\ref{lemma:bounded}.
\end{cor}

\begin{proof}
We observe that $\*M \nabla \*f(\*M \*y^\infty) = \frac{1}{n} \cdot 1_n \otimes \nabla f (\bar{y}^\infty)$ and hence $ \|\*M \nabla \*f(\*M \*y^\infty)\| = n^{-1} \|1_n \otimes \nabla f (\bar{y}^\infty)\|= (\sqrt{n})^{-1}\|\nabla f (\bar{y}^\infty)\|,$
where we obtain the last equality due to the fact that $\|1_n \otimes v\|^2 = n \|v\|^2$ for any vector $v$.

Moreover, $\*y^\infty$ is a critical point of~\eqref{eq:lyapunov_f} and therefore satisfies $\nabla \mL_t(\*y^\infty) = \*Z^t \nabla \*f(\*Z^t\*y^\infty) + \frac{1}{\alpha}\*Z^t\*y^\infty - \frac{1}{\alpha}\*Z^{2t}\*y^\infty = \mathbf{0}$.
From the double stochasticity of $\*Z$, multiplying the above relation with $\*M$ yields $\*M \nabla \mL_t(\*y^\infty) = \*M \nabla \*f(\*Z^t\*y^\infty) = \mathbf{0}$. After combining all the preceding results, we obtain,
\begin{equation*}
    \begin{split}
       \|\nabla f (\bar{x}^\infty)\| &= \sqrt{n} \|\*M\nabla \*f(\*M \*y^\infty) - \*M \nabla \*f(\*Z^t\*y^\infty)\|\\
       &\leq \sqrt{n} L \|\*M \*y^\infty - \*Z^t\*y^\infty\| \leq \beta^t \sqrt{n} L \|\*y^\infty\|,
    \end{split}
\end{equation*}
where used the spectral properties of $\*M$ and Assumption~\ref{assum:smoothness} to get the first inequality and the spectral properties of $\*Z$ to get the second inequality. Applying Lemma~\ref{lemma:bounded} yields the result of this Corollary.
\end{proof}

\subsection{Second order guarantees}


In this subsection, we provide second order guarantees for the NEAR-DGD$^t$ method. Specifically, using recent results stemming from dynamical systems theory, we will prove that NEAR-DGD$^t$ almost surely avoids the strict saddles of the Lyapunov function $\mL_t$ when initialized randomly. Hence, if $\mL_t$ satisfies the strict saddle property, NEAR-DGD$^t$ converges to minima of $\mL_t$ with probability $1$. We begin by listing a number of additional assumptions and definitions. 

\begin{assum}\textbf{(Differentiability)}
\label{assum:differ2}
The functions $\*f$ is $\mathcal{C}^2$.
\end{assum}
Assumption~\ref{assum:differ2} implies that the function $\mL_t$ is also $\mathcal{C}^2$.

\begin{defi}\textbf{(Differential of a mapping)}~[Ch.~3, \cite{AbsMahSep2008}]
\label{def:differ}
The differential of a mapping $g:\mathcal{X}\rightarrow \mathcal{X}$, denoted as $Dg(x)$, is a linear operator from $\mathcal{T} (x) \rightarrow \mathcal{T} (g(x))$, where $\mathcal{T} (x)$ is the tangent space of $\mathcal{X}$ at point $x$. Given a curve $\gamma$ in $\mathcal{X}$ with $\gamma(0)=x$ and $\frac{d \gamma}{dt}(0)= v \in \mathcal{T}(x)$, the linear operator is defined as $Dg(x)v = \frac{d(g \circ \gamma)}{dt}(0) \in \mathcal{T}(g(x))$. The determinant of the linear operator $\det(Dg(x))$ is the determinant of the matrix representing $Dg(x)$ with respect to an arbitrary basis.
\end{defi}

\begin{defi} 
\label{def:unstable}
\textbf{(Unstable fixed points)} The set of unstable fixed points $\mathcal{A}^\star_g$ of a mapping $g:\mathcal{X}\rightarrow \mathcal{X}$ is defined as $\mathcal{A}^\star_g = \{ x \in \mathcal{X}: g(x)=x, \max_i |\lambda_i(Dg(x))| > 1\}.$
\end{defi}

\begin{defi} 
\label{def:saddles}
\textbf{(Strict saddles)} The set of strict saddles $\mathcal{X}^\star$ of a function $f:\mathcal{X}\rightarrow \mathbb{R}$ is defined as $\mathcal{X}^\star = \{ x^\star \in \mathcal{X}: \nabla f(x^\star) = 0, \lambda_1(\nabla ^2 f(x^\star)) < 0\}.$
\end{defi}

We can express NEAR-DGD$^t$ as a mapping $g: \R^{np} \rightarrow \R^{np}$,
\begin{equation*}
\label{eq:near_dgd_map}
    g(\* y) = \*Z^t\*y - \alpha \nabla \*f(\*Z^t\*y),
\end{equation*}
with $ Dg(\*y)= \*Z^t \left(I- \alpha \nabla^2 \*f(\*Z^t\*y)\right)$.
Let us define the set of unstable fixed points $\mathcal{A}_g^\star$ of NEAR-DGD$^t$ and the set of strict saddles $\mathcal{Y}^\star$ of the Lyapunov function~\eqref{eq:lyapunov_f} following Def.~\ref{def:unstable} and~\ref{def:saddles}, respectively. Corollary~$1$ of~\cite{lee2017firstorder} implies that if $\det(Dg(\*y))\neq0$ for all $\*y \in \mathbb{R}^{np}$ and $\mathcal{Y}^\star \subset \mathcal{A}^\star_g$, then NEAR-DGD$^t$ almost surely avoids the strict saddles of~\eqref{eq:lyapunov_f}. We will show that this is indeed the case in Theorem~\ref{thm:2nd_order}.


\begin{thm}\textbf{(Convergence to 2nd order stationary points)} 
\label{thm:2nd_order}
Let $\{\*y_k\}$ be the sequence of iterates generated by NEAR-DGD$^t$ under steplength $\alpha < 1/L$. Then if the Lyapunov function $\mL_t$ satisfies the strict saddle property, $\{\*y_k\}$ converges almost surely to $2^{nd}$ order stationary points of $\mL_t$ under random initialization.
\end{thm}

\begin{proof}
We begin this proof by showing that $\det(Dg(\*y))\neq 0$ for every $\*y \in \mathbb{R}^{np}$. Let $\lambda_i(\nabla^2 \*f(\*Z^t\*y))$ be the eigenvalues of the Hessian $\nabla^2 \*f(\*Z^t\*y)$. Assumption~\ref{assum:smoothness} implies that $\lambda_i(\nabla^2 \*f(\*Z^t\*y)) < L$ for all $i\in \{1,...,np\}$. Using standard properties of the determinant, we obtain, $\det\left(Dg(\*x)\right) = \det(\*Z^t)\det(I - \alpha \nabla^2 \*f(\*Z^t\*y))= \left(\prod_i \lambda^t_i(\*Z) \right)\left(\prod_i (1-\alpha \lambda_i(\nabla^2 \*f(\*Z^t\*y))\right)$.
Thus, $\det\left(Dg(\*x)\right) \neq 0$ by the positive-definiteness of $\*Z$ and $\alpha < 1/L$.

We will now confirm that $\mathcal{Y}^\star \subset \mathcal{A}^\star_g$. Every critical point $\*y^\star$ of~\eqref{eq:lyapunov_f} satisfies $\nabla\mathcal{L}_t(\*y^\star)=\mathbf{0}$, namely $ \*Z^t \nabla \*f(\*Z^t\*y^\star) + \frac{1}{\alpha} \*Z^t\*y^\star - \frac{1}{\alpha} \*Z^{2t} \*y^\star = \mathbf{0}$.
Since $\*Z$ is positive-definite and by extension non-singular, we can multiply both sides of the equality above with $\alpha \*Z^{-t}$ and re-arrange the resulting terms to obtain $\*y^\star = g(\*y^\star)$. Finally, the Hessian of~\eqref{eq:lyapunov_f} at $\*y^\star$ is given by,
\begin{equation}
\label{eq:hessian}
    \begin{split}
        \nabla^2\mathcal{L}_t(\*y^\star) &= \*Z^t \nabla^2 \*f(\*Z^t\*y^\star) \*Z^t+ \frac{1}{\alpha} \*Z^t(I -\*Z^t)\\
        &=\frac{1}{\alpha}\left(I-Dg(\*y^\star)\right)\*Z^t.
    \end{split}
\end{equation}
We define the matrix $\*P:=\alpha \*Z^{-\frac{t}{2}}\nabla^2\mathcal{L}_t(\*y^\star) \*Z^{-\frac{t}{2}}$. Using the positive-definiteness of $\*Z$, we obtain from~\eqref{eq:hessian} $ I-Dg(\*y^\star)= \alpha \nabla^2\mathcal{L}_t(\*y^\star) \*Z^{-t} =  \*Z^{\frac{t}{2}} \*P \*Z^{-\frac{t}{2}},$
which implies that $\left(I-Dg(\*y^\star)\right)$ and $\*P$ are similar matrices and have identical spectrums. Moreover, the matrix $\*Z^{-\frac{t}{2}}$ is symmetric by Assumption~\ref{assum:symmetry}. Hence, $\*P$ and $\left(\alpha\nabla^2\mathcal{L}_t(\*y^\star)\right)$ are congruent and by Sylvester's law of inertia~[Theorem 4.5.8,~\cite{horn_matrix_2012}] they have the same number of negative eigenvalues. Given that $\nabla^2\mathcal{L}_t(\*y^\star)$ has at least one negative eigenvalue by Def.~\ref{def:saddles}, we conclude that so does $\*P$ and there exists index $i$ such that $1-\lambda_i(Dg(\*y^\star)) < 0 $ or $ \lambda_i(Dg(\*y^\star)) > 1$. Applying~[Corollary~$1$,~\cite{lee2017firstorder}] establishes the desired result.
\end{proof}
Before we conclude this section, we make one final remark on the asymptotic behavior of NEAR-DGD$^t$ as the parameter $t$ becomes large. 
\begin{cor}(\textbf{Convergence to SOS})
Let $\{\*x_k\}$ and $\{\*y_k\}$ be the sequences of NEAR-DGD$^t$ iterates produced by~\eqref{eq:near_dgd_x} and~\eqref{eq:near_dgd_y}, respectively, from initial point $\*y_0$ with $t(k)=t \in \mathbb{N}^+$ and steplength $\alpha < 1/L$. Moreover, suppose that $\*y^\infty$ is the limit point of NEAR-DGD$^t$ and let $\*x^\infty = \*Z^t \*y^\infty=[(x_1^\infty)',...,(x_n^\infty)']'$. Then $x_i^\infty=x_j^\infty$ for all $i\neq j$ and  $x_i^\infty$ approaches the $2^{nd}$ order stationary solutions (SOS) of Problem~\ref{eq:prob_orig} as $t\rightarrow \infty$ for all $i \in \mathcal{V}$.
\end{cor}
\begin{proof}
By Theorems~\ref{thm:global_conv} and~\ref{thm:2nd_order}, we have $\{\*y_k\} \rightarrow \*y^\infty$, where $\*y^\infty$ is a minimizer of $\mL_t$. Since $\*Z$ is non-singular, we also have $\{\*x_k\} \rightarrow \*x^\infty = \*Z^t \*y^\infty$. As $t\rightarrow \infty$, Lemmas~\ref{lemma:bounded_dist} and~\ref{lemma:dist_optimality} yield $\|x_i^\infty - \bar{x}^\infty\|\rightarrow 0$ and $\|\nabla f(\bar{x}^\infty)\| \rightarrow 0$, respectively, implying that $x_{i}^\infty$ and $\bar{x}^\infty$ approach each other and the critical points of $f$. Finally, $\nabla^2 \mL_t(\*y^\infty) \succeq 0$ by Theorem~\ref{thm:2nd_order}, where $\nabla^2 \mL_t(\*y^\infty)=\*Z^t \nabla^2 \*f(\*x^\infty)\*Z^t + \alpha^{-1}\*Z^t(I-\*Z^t)$. Multiplying this relation with the matrix $\*M$ on both sides, we obtain $\*M\nabla^2  \mL_t(\*y^\infty) \*M = \*M\nabla^2 \*f(\*x^\infty)\*M$. As $t\rightarrow \infty$, Lemma~\ref{lemma:bounded_dist} yields $\*M\nabla^2 \mL_t(\*y^\infty)\*M \rightarrow \*M\nabla^2 \*f(1_n \otimes \bar{x}^\infty)\*M=n^{-2} 1_n 1_n' \otimes \nabla^2 f(\bar{x}^\infty)$. Therefore, $\nabla^2 f(\bar{x}^\infty)\succeq 0$ by Sylvester's law of inertia for congruent matrices~[Theorem 4.5.8,~\cite{horn_matrix_2012}]. Based on the above, we conclude that NEAR-DGD$^t$ approaches the $2^{nd}$ order stationary solutions of Problem~\ref{eq:prob_orig} as $t\rightarrow \infty$.
\end{proof}

\section{Numerical Results}
\label{sec:numerical}
We evaluate the empirical performance of NEAR-DGD on the following regularized quadratic problem,
\begin{equation*}
    \min_{x \in \mathbb{R}^p} f(x) = \frac{1}{2}\sum_{i=1}^n \left(\|x\|^2_{Q^i}\right) + \frac{1}{4}\|x\|^4_{D_I},
\end{equation*}
where $I\in\{1,...,p\}$ is some positive index and $Q^i \in \mathbb{R}^{p \times p}$ and $D_I\in \mathbb{R}^{p \times p}$ are diagonal matrices constructed as follows: $q^i_{jj} < 0$ if $j= I$ and $q^i_{jj} > 0$ otherwise, and $D_I = c \cdot e_I e_I'$, where $c>0$ is a constant and $e_I$ is the indicator vector for the $I^{th}$ element. It is easy to check that $f$ has a unique saddle point at $x=\mathbf{0}$ and two minima at $x^\star=\pm\frac{1}{c} \left(\sqrt{\sum_{i=1}^n -q^i_{II}}\right)e_I$. We can distribute this problem to $n$ nodes by setting $f_i(x)=\frac{1}{2}\|x\|^2_{Q^i} + \frac{1}{4n}\|x\|^4_{D_I}$. Moreover, each $f_i$ has Lipschitz gradients in any compact subset of $\mathbb{R}^{p}$. 

We set $p=I=4$ for the purposes of our numerical experiment. The matrices $Q^i$ were constructed randomly with $q^i_{jj} \in (-1,0)$ for $j = I$ and $q^i_{jj} \in (0,1)$ otherwise, and the parameter $c$ of matrix $D_I$ was set to $1$. We allocated each $f_i$ to a unique node in a network of size $n=12$ with ring graph topology. We tested $6$ methods in total, including DGD~\cite{NedicSubgradientConsensus,zeng_nonconvex_2018}, DOGT (with doubly stochastic consensus matrix)~\cite{daneshmand_second-order_2020}, and  $4$ variants of the NEAR-DGD method: $i$) NEAR-DGD$^1$ (one consensus round per gradient evaluation), $ii$) NEAR-DGD$^5$ ($5$ consensus rounds per gradient evaluation), $iii$) a variant of NEAR-DGD where the sequence of consensus rounds increases by $1$ at every iteration, and to which we will refer as NEAR-DGD$^+$, and $iv$) a practical variant of NEAR-DGD$^+$, where starting from one consensus round/iteration, we double the number of consensus rounds every $100$ gradient evaluations. We will refer to this last modification as NEAR-DGD$^+_{100}$. All algorithms were initialized from the same randomly chosen point in the interval $[-1,1]^{np}$. The stepsize was manually tuned to $\alpha=10^{-1}$ for all methods.

In Fig.~\ref{fig:errors}, we plot the objective function error $f(\bar{x}_k)-f^\star$ where $f^\star = f(x^\star)$ (Fig.~\ref{fig:error_f}) and the distance $\|\bar{x}_k\|$ of the average iterates to the saddle point $x=\mathbf{0}$ (Fig.~\ref{fig:error_0}) versus the number of iterations/gradient evaluations  for all methods. In Fig.~\ref{fig:error_f}, we observe that convergence accuracy increases with the value of the parameter $t$ of NEAR-DGD$^t$, as predicted by our theoretical results. NEAR-DGD$^1$ performs comparably to DGD, while the two variants of NEAR-DGD paired with increasing sequences of consensus rounds per iteration, i.e. NEAR-DGD$^+$ and NEAR-DGD$^+_{100}$, achieve exact convergence to the optimal value with faster rates compared to NEXT. All methods successfully escape the saddle point of $f$ with approximately the same speed~(Fig.~\ref{fig:error_0}). We noticed that the trends in Fig.~\ref{fig:error_0} were very sensitive to small changes in problem parameters and the selection of initial point. 

In Fig.~\ref{fig:cost}, we plot the objective function error $f(\bar{x}_k)-f^\star$ versus the cumulative application cost (per node) for all methods, where we calculated the cost per iteration using the framework proposed in~\cite{berahas_balancing_2019},
\begin{equation*}
\label{eq:cost_frame}
    \text{Cost} = c_c \times \# \text{Communications} +  c_g \times \# \text{Computations},
\end{equation*}
where $c_c$ and $c_g$ are constants representing the application-specific costs of one communication and one computation operation, respectively. In Fig.~\ref{fig:cost_cheap}, the costs of communication and computation are equal ($c_c=c_g$) and NEXT outperforms NEAR-DGD$^+$ and NEAR-DGD$^+_{100}$ since it requires only two communication rounds per gradient evaluation to achieve exact convergence. Conversely, in Fig.~\ref{fig:cost_exp}, the cost of communication is relatively low compared to the cost of computation ($c_c = 10^{-2} c_g$). In this case, NEAR-DGD$^+$ converges to the optimal value faster than the remaining methods in terms of total application cost.

\begin{figure}[!t]
\centering
\subfloat[Objective function error]{\includegraphics[width=.24\textwidth]{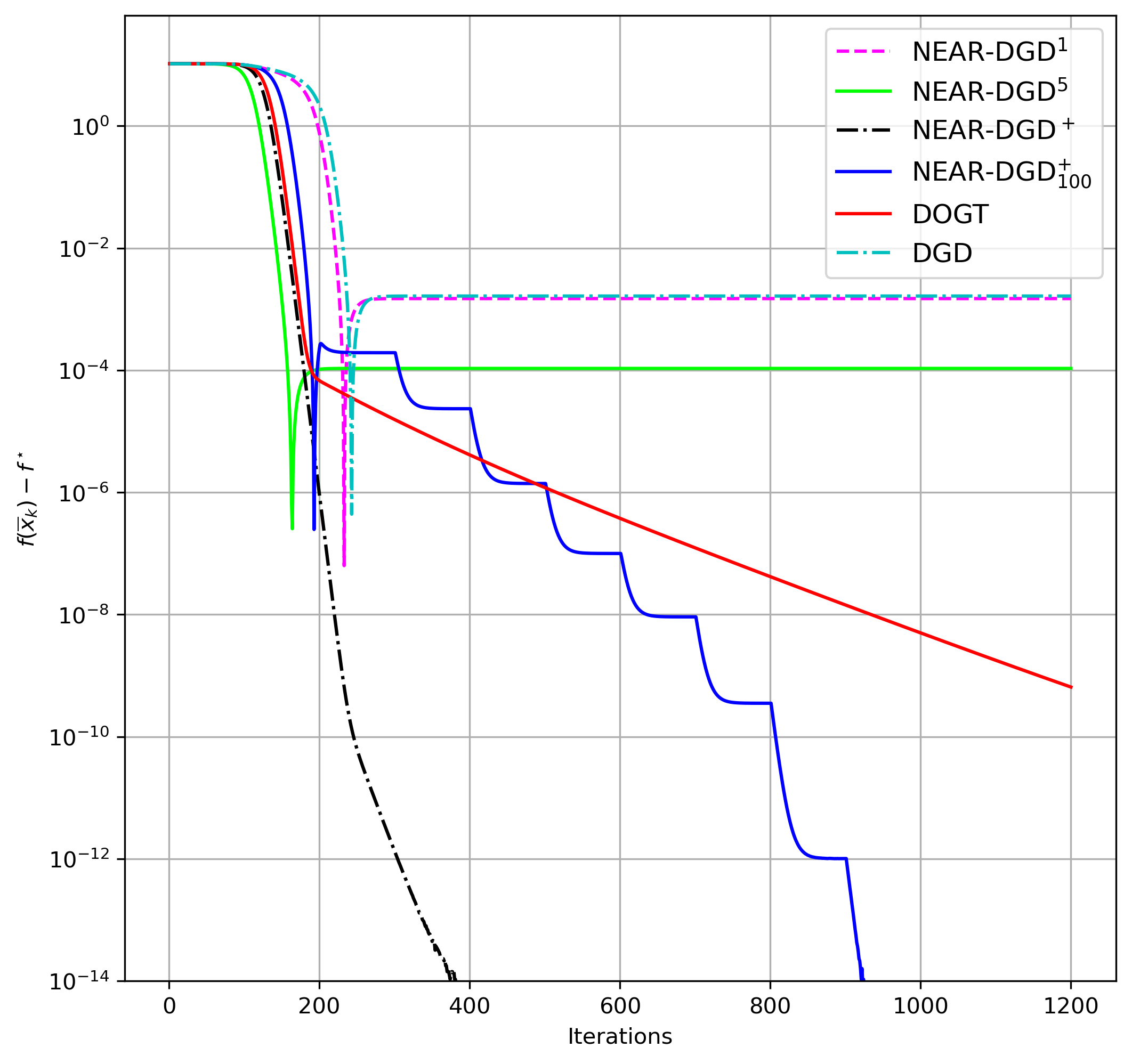}%
\label{fig:error_f}}
\hfil
\subfloat[Distance to $x=\mathbf{0}$ (saddle)]{\includegraphics[width=.24\textwidth]{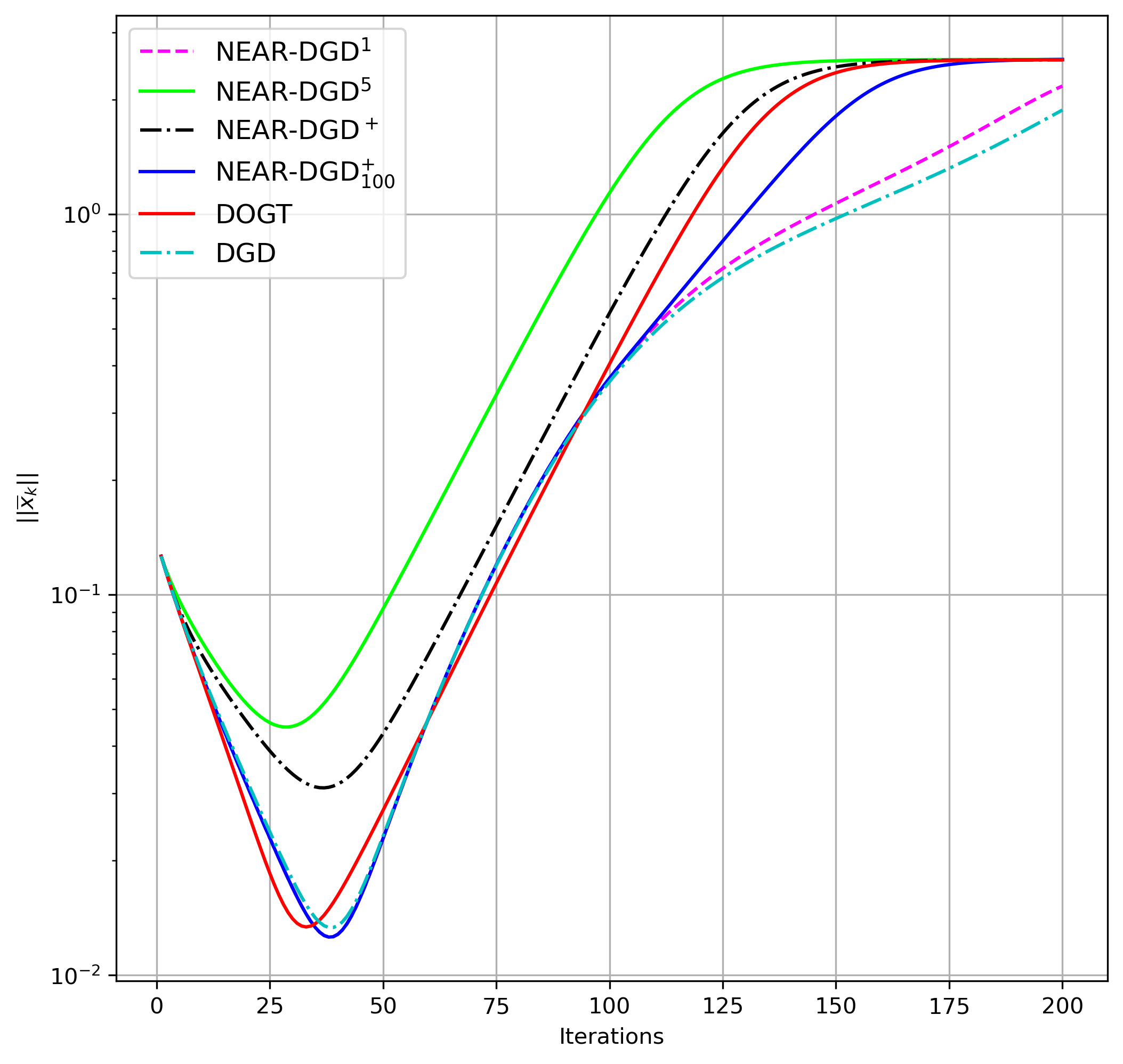}%
\label{fig:error_0}}
\caption{Distance to $f^\star$ (left) and to saddle point (right) }
\label{fig:errors}
\end{figure}

\begin{figure}[!t]
\centering
\subfloat[$c_g=1$, $c_c=1$]{\includegraphics[width=.24\textwidth]{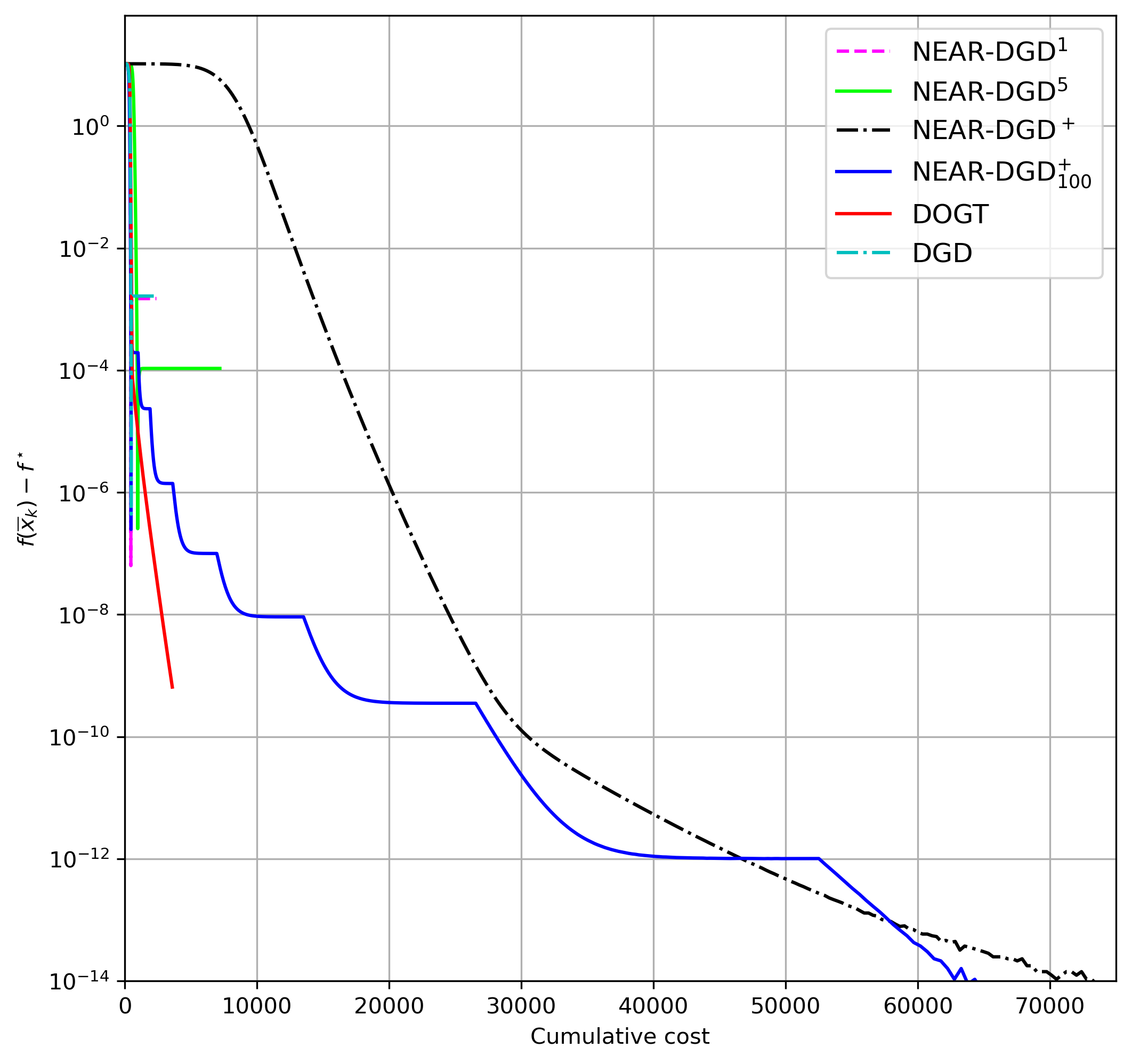}%
\label{fig:cost_cheap}}
\hfil
\subfloat[$c_g=1$, $c_c=10^{-2}$]{\includegraphics[width=.24\textwidth]{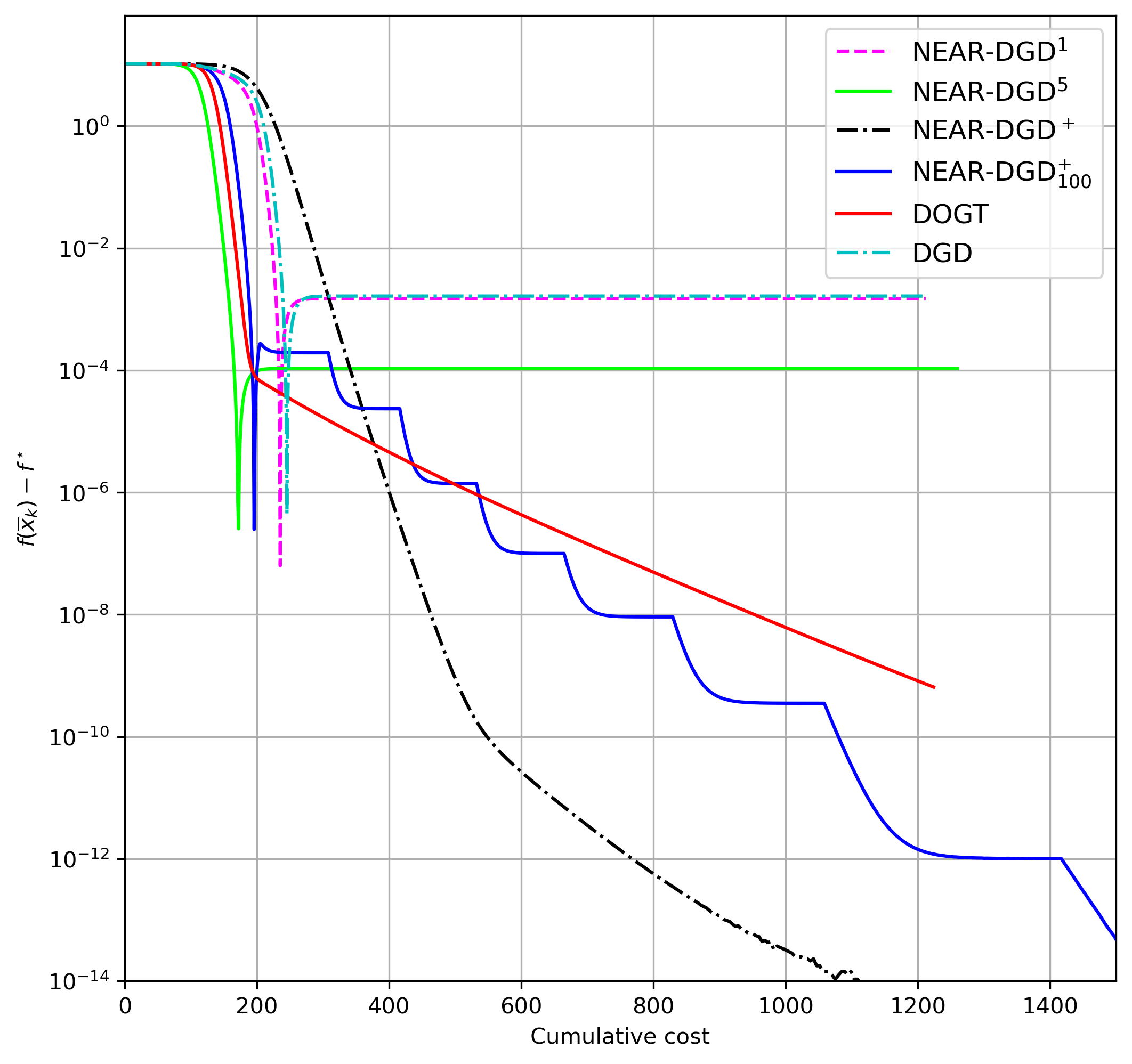}%
\label{fig:cost_exp}}
\caption{Objective function error as a function of cumulative application cost (per node)}
\label{fig:cost}
\end{figure}

\section{Conclusion}
\label{sec:conclusion}
NEAR-DGD~\cite{berahas_balancing_2019} is a distributed first order method that permits adjusting the amounts of computation and communication carried out at each iteration to balance convergence accuracy and total application cost. We have extended to the nonconvex setting the analysis of NEAR-DGD$^t$, a variant of NEAR-DGD performing a fixed number of communication rounds at every iteration, which is controlled by the parameter $t$. Given a connected, undirected network with general topology, we have shown that NEAR-DGD$^t$ converges to minimizers of a customly designed Lyapunov function and locally approaches the minimizers of the original objective function as $t$ becomes large. Our numerical results confirm our theoretical analysis and indicate that NEAR-DGD can achieve exact convergence to the $2^{nd}$ order stationary points of Problem~\eqref{eq:prob_orig} when the number of consensus rounds increases over time.






\bibliographystyle{IEEEtran}
\bibliography{ref, nonconvex, KL, stochastic}




\end{document}